\documentclass[12pt]{amsart}
 \usepackage{amsmath,amsfonts,amsthm,amscd,textcomp,times, }
  \ifx\pdfoutput\undefined
   \usepackage[dvips]{graphicx}
   \else
   \usepackage[pdftex]{graphicx}
   \usepackage[all,cmtip]{xy}
\usepackage[T1]{fontenc}
\usepackage{calligra}
\usepackage{multido}
\usepackage{frcursive}
\usepackage{mathrsfs}
\usepackage{tikz}
   \fi
  \usepackage{epstopdf}
\usepackage{cite}
 \newtheorem{theorem}{Theorem}

\newtheorem{proposition}{Proposition}
\newtheorem{lemma}{Lemma}

\newtheorem{corollary}{Corollary}
\newtheorem{remark}{Remark}

\numberwithin{equation}{section}
\numberwithin{theorem}{section}
\numberwithin{proposition}{section}
\numberwithin{lemma}{section}
\numberwithin{claim}{section}
\numberwithin{corollary}{section}

\newcommand{\bull}{\ensuremath{{}\bullet{}}}
 \newcommand{\gr}{\ensuremath{\mathbb{G}(N-n-1, \cpn)}}
\newcommand{\cpn}{\ensuremath{\mathbb{P}^{N}}}
\newcommand{\slnc}{\ensuremath{SL(N+1,\mathbb{C})}}
\newcommand{\dlb}{\ensuremath{\overline{\partial}}}
\newcommand{\dl}{\ensuremath{\partial}} 
\newcommand{\ra}{\ensuremath{\longrightarrow}}
\newcommand{\om}{\ensuremath{\omega}}

\newcommand{\xhyp}{\ensuremath{X\times\mathbb{P}^{n-1}}}
 
\begin{document}
  \DeclareGraphicsExtensions{.pdf,.gif,.jpg}
\title[K-Stability]{Fourier-Mukai Transforms, Euler-Green Currents, and K-Stability }
\author{Sean Timothy Paul and Kyriakos Sergiou}
\email{stpaul@wisc.edu}
\email{sergiou@wisc.edu}
\address{Mathematics Department at the University of Wisconsin, Madison}
\date{April 30, 2019} 
 \vspace{-5mm}  
\begin{abstract} 
We provide an analog  of the Hilbert-Chow morphism for generalized discriminants. As an application we recover the Main Lemma from \cite{paul2012} as well as Theorem 0.1 from \cite{dingtian}. 
 Our work also exhibits a wide range of energy functionals in K\"ahler geometry as Fourier-Mukai transforms. Consequently these energies are completely determined by dual type varieties and therefore have logarithmic singularities when restricted to the space of algebraic potentials. 
 This work was inspired by the many ideas of Gang Tian concerning canonical K\"ahler metrics and the stability of projective algebraic varieties .
\end{abstract}
\maketitle 
\tableofcontents 
 \newpage
\section{\ \ Introduction and Statement of Results}
Let $\mathbb{X}\xrightarrow{\pi} B$ be a flat (relative dimension $n$) family of smooth polarized, linearly normal, complex subvarieties of some fixed $\mathbb{P}^N$ parametrized by a reduced and irreducible (quasi) projective base $B$. We do not assume that $B$ is smooth.  
 Let $Q$ be a locally free sheaf of rank $n+1$ over $\mathbb{X}$.
We assume that \
\ \\
\begin{align*}
 &(**) \mbox{ There is a finite dimensional subspace  $ W\subset \Gamma(\mathbb{X}\ ,\ Q)$ generating each fiber .}\  (**)  \  \\
\end{align*}
This is equivalent to an  exact sequence of vector bundles over $\mathbb{X}$ 
\begin{align*}
0\ra\mathcal{S}\ra \mathbb{X}\times W\ra Q\ra 0 \  
\end{align*}
where the sub bundle  $\mathcal{S}$ is the kernel of the evaluation map $\mathbb{X}\times W\ra Q$ .
 
 This defines 
\begin{align*}
&\bull  \mathcal{E}:=  Q\otimes  \mathcal{O}_{\mathbb{P}(W)}(1)\\
\ \\
&\bull  \Psi \in \Gamma(\mathbb{X}\times \mathbb{P}(W), Q\otimes\mathcal{O}_{\mathbb{P}(W)}(1)) \quad \mbox{ (the evaluation section)}\\
 \ \\
&\bull   I :=\mathbb{P}(\mathcal{S}) =\left( \Psi= 0\right) \subset \mathbb{X}\times \mathbb{P}(W)  \\
\ \\
& \bull Z:= \pi(I ) 
\end{align*}
By abuse of notation we also denote by $\pi$ the projection  
\begin{align*}
 \pi:\mathbb{X}\times \mathbb{P}(W)\ra B\times \mathbb{P}(W)  
\end{align*}
The situation may be visualized as follows
\begin{align*} 
\xymatrix{&\mathcal{E}\ar[d]\\  
	I \ar@{^{(}->}[r]^-{\iota} \ar[d] &  \mathbb{X}\times\mathbb{P}(W) \ar@/^1pc/[u]^{ {\Psi}} \ar[d]^{\pi}&\\  
	{Z}\ar@{^{(}->}[r]\ar[d]&B\times \mathbb{P}(W)  \ar[d]\\
	B\ar[r]&B }
\end{align*}
Observe that the  {codimension} of $I$ in $\mathbb{X}\times \mathbb{P}(W)$ is equal to $n+1$.

  The \textbf{basic assumption} in this paper is that $I\overset{\pi} \ra Z$ is a \emph{simultaneous resolution of singularities}.
This means:
\begin{align*}
&i)\ \mbox{ $I\overset{\pi}\ra Z$ is birational .}\\
\ \\
&ii)\ \mbox{For each $b\in B$}\\
&\mbox{$I_b:=\mathbb{P}(\mathcal{S}|_{\mathbb{X}_b})\overset{\pi}\ra  Z_b:=Z\cap(\{b\}\times\mathbb{P}(W))$ is also birational.}
\end{align*}
Our basic assumption implies that the codimension of $Z_b$ in $\mathbb{P}(W)$ is one. All of this data will be called a \emph{basic set up} for the family $\mathbb{X}\ra B$ .

Next we observe that the degrees of the divisors $Z_b$ are constant in $b\in B$. This follows at once from the identity
\begin{align*}
\deg(Z_b)=\int_{\mathbb{X}_b}c_n(Q|_{\mathbb{X}_b})>0 \ .
\end{align*}
We will denote this common degree by $\widehat{d}$
\begin{align*}
\deg(Z_b)=\widehat{d} \quad \mbox{for all $b\in B$}\ .
\end{align*}

Therefore, for each $b\in B$ there is a unique (up to scale) irreducible polynomial $f_b$
\begin{align*}
 f_b\in\mathbb{C}_{\widehat{d}}\ [W]\setminus\{0\}
\end{align*}
such that 
\begin{align*}
\left(f_b=0\right)=Z_b\ .
 \end{align*}
We may therefore define a map
\begin{align*}
\Delta:B\ra |\mathcal{O}(\widehat{d})|:=\mathbb{P}({\mathbb{C}}_{\widehat{d}}\ [W]) \quad \Delta(b):=[f_b] \ .
\end{align*}

The main result of this paper is the following
\begin{theorem}\label{hilb-disc}
	$\Delta$ is a morphism of quasi-projective varieties. In particular, $\Delta$ is holomorphic.
\end{theorem}
\subsection{Hermitean metrics and Base Change}
Now we introduce Hermitean metrics on everything via  ``base change''. Let $S$ be a smooth complex algebraic variety together with a morphism to the base of our family $\mathbb{X}\overset{\pi}\ra B$
\begin{align*}
S\ra B\ .
\end{align*}

The basic set up over $B$ may be pulled back to a basic set up over $S$. The corresponding morphism will still be denoted by $\Delta$. Our first assumption on $S$ is that 
\begin{align*}
\Delta^*\mathcal{O}(1)\cong\mathcal{O}_S\qquad \mathcal{O}(1):=\mbox{the hyperplane over $|\mathcal{O}(\widehat{d})|$} \ .
\end{align*}

Equivalently we may ``lift'' the map $\Delta $ the the cone over $|\mathcal{O}(\widehat{d})|$
\begin{align*}
 \xymatrix{& {\mathbb{C}}_{\widehat{d}}\ [W]\setminus\{0\}\ar[d]\\
 S\ar[ru] \ar[r]^\Delta&|\mathcal{O}(\widehat{d})|
}
\end{align*}
By abuse of notation the lifted map will also be denoted by $\Delta$. Next we fix a positive definite Hermitean form $<\cdot\ ,\ \cdot>$ on $W$. We let $\mathcal{H}^{+}(W, <\cdot\ ,\ \cdot>)$ denote self adjoint positive linear maps on $W$. This parametrises all postitive Hermitian forms on $W$. Let $h$ be a smooth map
\begin{align*}
h:S\ra \mathcal{H}^{+}(W, <\cdot\ ,\ \cdot>)\ .
\end{align*}
Using this map we define a ``dynamic'' metric $H_S$ on the trivial bundle $\mathbb{X}_S\times W$
\begin{align*}
H_S(p)<w_1,w_2>:=<h(\pi(p))w_1,w_2>\ .
\end{align*}
Similarly we have the ``static'' metric 
\begin{align*}
H_0(p)<w_1,w_2>:=< w_1,w_2>\ .
\end{align*}
These metrics descend to metrics $H^Q_S$ and $H^Q_0$ on $Q$. 
The Hermitean inner product on $W$ induces a Fubini Study metric $h^m_{FS}$ on $\mathcal{O}_{\mathbb{P}(W)}(m)$ for any $m\in \mathbb{Z}$ as well as a K\"ahler form $\om_{FS}$ on $\mathbb{P}(W)$ . This will be fixed throughout the paper. 
Tensoring $H^Q_S$ and $H^Q_0 $ with $h_{FS}$ gives two metrics $H^{\mathcal{E}}_S$ and $H^{\mathcal{E}}_0$on $\mathcal{E}$.

We need the following result from \cite{bgs2007}.
\begin{theorem}  {\emph{(Bismut, Gillet, Soul\'e)}\newline  Let $\mathcal{E}\overset{\pi}{\ra}M$ be a rank $r$ holomorphic vector bundle over a complex manifold $M$. Given any Hermitean metric $H^{\mathcal{E}}$ on $\mathcal{E}$  
	there exists a current $\hat {\mathbf{e}}(H^{\mathcal{E}})\in \mathcal{D}^{'}(\mathcal{E})$  whose wave front set is included in ${\mathcal{E}_{\mathbb{R}}^{*}}  $ and which satisfies the following equation of currents on }$\mathcal{E}$
	\begin{align*}
	\frac{\sqrt{-1}}{2\pi} \dl\dlb\hat {\mathbf{e}}(H^{\mathcal{E}})=\delta_{M}- \pi^{*}{\mathbf{e}}(H^{\mathcal{E}})\ .
	\end{align*}
 {The current $\hat {\mathbf{e}}(H^{\mathcal{E}})$ can be pulled back to $M$ by any section $ {\mathbf{s}}$ and satisfies}
	\begin{align*}
	\frac{\sqrt{-1}}{2\pi} \dl\dlb{\mathbf{s}}^{*}\hat {\mathbf{e}}(H^{\mathcal{E}})=\delta_{Z(\mathbf{s})}-{\mathbf{e}}(H^{\mathcal{E}}) \ 
	\end{align*}
	 {where $\mathbf{e}(H^{\mathcal{E}})\in C^{\infty}(\Lambda^{r,r}_M)$ is the Euler form of $\mathcal{E}$ in Chern-Weil theory. Moreover, given two Hermitean metrics $H_0^{\mathcal{E}}$ and $H_1^{\mathcal{E}}$
the difference of the corresponding currents is smooth and up to $\dl$ and $\dlb$ terms is given by}
\begin{align*}
{\mathbf{s}}^{*}\hat {\mathbf{e}}(H_0^{\mathcal{E}})-{\mathbf{s}}^{*}\hat {\mathbf{e}}(H_1^{\mathcal{E}})=- {\mathbf{e}}(H_{1}^{\mathcal{E}}, H_{0}^{\mathcal{E}})\ 
\end{align*}
 {where ${\mathbf{e}}(H_{1}^{\mathcal{E}} , H_{0}^{\mathcal{E}})$ denotes the Bott-Chern double transgression of the Euler form with respect the given metrics.} 
\end{theorem}

We apply this result to our situation. $M=\mathbb{X}\times \mathbb{P}(W)$ , $\mathcal{E}=\mathcal{Q}\otimes\mathcal{O}(1)$ , $\mathbf{s}=\Psi$ so that ${Z(\mathbf{s})}=I$.

We assume that our two metrics $H^{\mathcal{E}}_0$ and $H_S^{\mathcal{E}}$ satisfy the following conditions \footnote{Condition A2 comes from Lemma 4.1 on pg. 278 of \cite{paul2012}.} :
\begin{align*}
&\mbox{A1.}\ \frac{\sqrt{-1}}{2\pi}\dl\dlb\Psi^*\hat{\mathbf{e}}( H^{\mathcal{E}}_0)\wedge\om_{FS}^l=\delta_{I}\wedge\om_{FS}^l \quad l+1:=\dim(W)\\
\ \\
&\mbox{A2. The function}\ S\ni s\mapsto \int_{\mathbb{X}_{s}\times\mathbb{P}(W)}{\Psi}^{*}\hat{ \mathbf{e}}( H_S^{\mathcal{E}})\wedge\om_{FS}^l \
\mbox{is pluriharmonic .} 
\end{align*}
We need one more ingredient to state our next result. 
 Let $d\in\mathbb{Z}_{>0}$ .  Recall that the Mahler measure is defined by
 \begin{align*}
 \Theta :|\mathcal{O}( {d})|\ra \mathbb{R} \qquad \Theta([f]):= \int_{\mathbb{P}(W)}\log\frac{|f|^2_{h_{FS}^{d}}(\cdot)}{||f||_2^2}\om^{l}_{FS} \ .
 \end{align*}
 Tian has shown \cite{tian97} that
 \begin{proposition}
 	$\Theta$ is H\" older continuous , in particular it is bounded .
 \end{proposition}
Finally define $\theta(s):=\Theta(\Delta(s))$, and let $o$ be a basepoint in $S$ such that $H^{\mathcal{E}}_S|_{\mathbb{X}_o}=H^{\mathcal{E}}_0|_{\mathbb{X}_o}$\footnote {In our case this comes from assuming that $h(o)=Id_{W}\in \mathcal{H}^{+}(W, <\cdot\ ,\ \cdot>)$.}.\\
\ \\ 
The following corollary of Theorem \ref{hilb-disc} extends several ideas of Gang Tian (especially \cite{kenhyp}, \cite{bottchrnfrms}, and \cite{tian97} ) as well as the first author \cite{paul2012} .
\begin{corollary}\label{metrics} 
  The function
\begin{align} 
S\ni s\mapsto \log\left(e^{\theta(s)} \frac{||\Delta(s)||_2^2}{||\Delta( o )||_2^2}\right)-\int_{\mathbb{X}_s\times\mathbb{P}(W)} \mathbf{e}(H^{\mathcal{E}}_{0},{H}^{\mathcal{E}}_{S})\wedge \om_{FS}^l 
\end{align} \\
\ \\
is pluriharmonic . Therefore, if $\pi_1(S)=\{1\}$ there is an entire function $f:S\ra \mathbb{C}$ such that \\
\ \\
\begin{align}\label{pluriharmonic}
\log\left(e^{\theta(s)} \frac{||\Delta(s)||_2^2}{||\Delta( o )||_2^2}\right)=\int_{\mathbb{X}_s\times\mathbb{P}(W)} \mathbf{e}(H^{\mathcal{E}}_{0},{H}^{\mathcal{E}}_{S})\wedge \om_{FS}^l +\log|f(s)|^2 \ .
\end{align}
\end{corollary}
\begin{remark}
\emph{In all applications $f(s)$ is {constant} .}
\end{remark}
\subsection{$G$-equivariant Families and Bergman metrics}
Let $G$ be a reductive complex algebraic group. The family $\mathbb{X}\ra B$ is said to be $G$ equivariant provided that there exists
\begin{itemize}
\item A rational representation $G\ra \mathbb{C}^{N+1}$ \\
\ \\
\item A rational action $G\ra \mbox{Aut}(B)$ such that 
\begin{align*}
\sigma\cdot \mathbb{X}_b=\mathbb{X}_{\sigma\cdot b}
\end{align*}
 \end{itemize}
As the reader can imagine a $G$ equivariant basic set up for this family requires an action of $G$ on $Q$ which preserves $W$. This gives another (rational) representation of $G$
\begin{align*}
\rho_{W}:G \ra \mbox{Aut}(W) \ .
\end{align*} 
It is not hard to see that the canonical section $\Psi$ is $G$ invariant, therefore we have that
\begin{itemize}
\item $\sigma\cdot Z_{b}=Z_{\sigma \cdot b}$ for   $\sigma\in G $ . \\
\item $\sigma \cdot f_{b}=f_{\sigma\cdot b}$ .
\end{itemize}
 As a consequence the map $\Delta: B\ra |\mathcal{O}(d)|$ is $G$ equivariant. Now choose any basepoint $o\in B$ and consider the mapping 
 \begin{align*}
 G\ni\sigma \ra \sigma\cdot o\in B
 \end{align*}
$G$  now becomes our choice of $S$.  Observe that the change-of-base family $\mathbb{X}_G$ is given concretely by
 \begin{align*}
 \mathbb{X}_G=\{ (\sigma , y)\in G\times \cpn |\ y\in \mathbb{X}_{\sigma\cdot o} \} \cong G\times \mathbb{X}_{ o}  .
 \end{align*}

As in the previous section a fixed choice of Hermitian metric $< \cdot , \cdot >$ on $W$ provides us with a mapping
\begin{align*}
h:G\ni \sigma\ra \rho_{W}(\sigma)^{\dagger}\rho_{W}(\sigma)\in \mathcal{H}^{+}(W, <\cdot\ ,\ \cdot>)\ .
\end{align*}
$h$ induces a dynamic metric on $\mathbb{X}_G\times W$ 
\begin{align*}
 W\ni w_1 ,w_2 \ra < \rho_{W}(\sigma)w_1,\rho_{W}(\sigma)w_2>\ .
\end{align*}
 The static metric is taken to be $h(e)$, the same as in the previous section.   We assume that the metrics induced on $\mathcal{E}$ satisfy axioms A1. and A2. Now we restate corollary \ref{pluriharmonic} in the present special case, assuming the entire function $f:G\ra \mathbb{C}$ is constant.
\begin{corollary}\label{Gpluriharmonic} Assume $G$ is simply connected. Then the following identity holds for any fiber $\mathbb{X}_{o}$ of the family $\mathbb{X}\ra B$
\begin{align} 
\log\left(e^{\theta(\sigma)-\theta(e)} \frac{||\sigma\cdot\Delta(o)||_2^2}{||\Delta( o )||_2^2}\right)=\int_{\mathbb{X}_o\times\mathbb{P}(W)} \mathbf{e}(H^{\mathcal{E}}(e),{H}^{\mathcal{E}}(\sigma))\wedge \om_{FS}^l  \ .
\end{align}
 \end{corollary} 
The right hand side of \ref{Gpluriharmonic} can be expressed as an ``action functional'' $E_{\om}$ restricted to the Bergman potentials $\varphi_{\sigma}$ associated to the embedding  $ \mathbb{X}_o\ra \cpn $.

\begin{corollary}\label{asymptotics} Let $\lambda:\mathbb{C}^*\ra G$ be a one parameter subgroup of $G$, then 
\begin{align} 
E_{\om}(\varphi_{\lambda(t)}) =w_{\lambda}(\Delta(o))\log|t|^2+O(1) \quad |t|<< 1 \ ,
\end{align}
where $w_{\lambda}(\Delta(o))\in \mathbb{Z}$ denotes the weight \footnote{Mumford called this the \emph{slope} in \cite{git} .} of $\Delta(o)$ with respect to $\lambda$ . Assume that
\begin{align*}
\infty:=\lambda(0)\cdot o \in B \ ,
\end{align*}
with corresponding fiber $\mathbb{X}_{\infty}\subset \cpn$. Then
\begin{align}\label{genfutaki}
E_{\om}(\varphi_{\lambda(t)}) =w_{\lambda}(\Delta(\infty))\log|t|^2+O(1) \quad |t|<< 1 \ .
\end{align}
\end{corollary}
The important point in (\ref{genfutaki}) is that
\begin{align}\label{limitcycle}
\lim_{t\ra 0}\lambda(t)\cdot \Delta(o)=\Delta(\infty) \ ,
\end{align}
which follows at once from Theorem \ref{hilb-disc} .
  \begin{remark}
 \emph{(\ref{limitcycle}) can be used to recover a well known result of Ding and Tian (see the claim just below (3.15) on pg. 328 of \cite{dingtian} ) and was our initial motivation for writing this article.}
  \end{remark}

To summarize, $\Delta(\sigma\cdot o)$ is for each $\sigma \in G$ a polynomial on $W$ and $E_{\om}(\varphi_{\sigma})$ is a generalization of the Mabuchi K-energy map. For special choices of $Q$ it \emph{is} the Mabuchi energy up to lower order terms.  Our work suggests that most action functionals $E_{\om}$ of  interest on a polarized manifold are induced by the basic set up and in particular have logarithmic singularities along $G$.   In particular, along algebraic one parameter subgroups $\lambda(t)$ of $G$ the asymptotic expansion of the integral of $E_{\om}(\varphi_{\lambda(t)})$ as $t\ra 0$ is determined by the limiting behavior of the \emph{coefficients} of $\Delta(\lambda(t)\cdot o)$ as $t\ra 0$  \footnote{For a concrete example of this see Theorem B of \cite{paul2012} .}. \\
\ \\
\emph{Organization.} $\Delta$ is a (homogeneous) polynomial with zero set $Z\subset \mathbb{P}(W) $ say. In the second section we discuss the classical setting of how such $Z$'s arise. They are objects of elimination theory. In the third section we recall in detail an essential and rather ingenious idea of Cayley that allows us to \emph{construct} $\Delta$ from $Z$. Usually one can easily detect whether or not $Z$ has codimension one but it is very hard to find an explicit defining polynomial . This is the main problem of classical elimination theory. In the fourth section we use a vast generalization of Cayley's idea due to Grothendieck, Knudsen, and Mumford \cite{detdiv} which extends Cayley's construction from a \emph{fixed} variety to the more natural setting of a \emph{family} of varieties over some base $S$, this is where the Fourier-Mukai Transform makes an appearance. In the final section we prove Corollary \ref{metrics} by introducing metrics on everything and invoking the currents of Bismut , Gillet, and Soul\'e  on the one hand, and  the Poincar\'e Lelong formula on the other.
\section{\ \ Classical Elimination Theory}
 Let $X^n\hookrightarrow \cpn$ be a linearly normal subvariety. We consider a parameter space $\mathcal{F}$ of ``linear sub-objects'' $f$ of $\cpn$. That is, $\mathcal{F}$ may parametrize points, lines, planes, flags, etc. of $\cpn$.
Consider the admittedly vague statement
\begin{center}$(**)$ ``\emph{The general member $f$ of $\mathcal{F}$ has a certain order of contact along ${X}$}''   \end{center}
Let $Z$ be the (proper) subvariety of those $f\in\mathcal{S}$ \emph{violating} $(**)$. \newline
\ \\
\noindent\textbf{Example 1. } (\emph{Cayley-Chow Forms}) 
We take $\mathcal{S}$ to be the Grassmannian of $N-n-1$ dimensional linear subspaces of $\cpn$
\[\mathcal{F}=\gr:=\{L\subset \mathbb{C}^{N+1}\ | \ \dim(L)=N-n-1\} \ . \]
In this case  the general member $f=L$ of $\mathcal{S}$ \emph{fails} to meet $X$. Therefore we take
\[ Z=\{L\in \gr\ | L\cap X \neq \emptyset\}\ .\]
\ \\
\noindent\textbf{Example 2.} (\emph{Dual Varieties}) 
Let $\mathcal{F}=\check{\mathbb{P}}^N$ the dual projective space parametrizing hyperplanes $H$ inside $\cpn$. Bertini's theorem provides us with our ``order of contact'' condition. Namely, for generic $H\in \check{\mathbb{P}}^N$ $H\cap X$ is \emph{smooth}.
Therefore we define
\[ Z=\{ H\in \check{\mathbb{P}}^N\ | \ H\cap X \ \mbox{is singular} \ \} .\ \] 
Observe that
\[ Z=\{ H\in \check{\mathbb{P}}^N\ | \ \mbox{there exists}\ p\in X\ \mbox{such that}\ \mathbb{T}_p(X)\subset H\} . \]

 The next example includes both \textbf{1} and \textbf{2} as extreme cases.\\
\ \\
\noindent\textbf{Example 3.} (\emph{Higher associated hypersurfaces} (see \cite{gkz} pg. 104))\newline
Fix $L\subset \mathbb{C}^{N+1}$ , $\dim(L)=n+1<N+1$. Consider the subset $\mathcal{U}$ of the Grassmannian defined by
\begin{align*}
\begin{split}
 \mathcal{U}:=\{E\in G(r;\ \mathbb{C}^{N+1})|\  H^{\bull}\left(0\ra E\cap L\ra E \overset{\pi_{L}}{\ra}\mathbb{C}^{N+1}/L\ra 0 \right)=0\}\ .
 \end{split}
 \end{align*}
In order that $\mathcal{U}$ be open and dense in $G(r;\ \mathbb{C}^{N+1})$ it is enough to have 
\begin{align*}
r=\dim(E)\geq N+1-(n+1)=N-n .
\end{align*}
Therefore we set $r=N-l$ where $0\leq l\leq n$. 
By the rank plus nullity theorem of linear algebra we have
\begin{align*}
\dim(E\cap L)+\dim(\pi_{L}(E)=\dim(E)=N-l \ .
\end{align*}
Therfore $E\in Z:=G(r;\ \mathbb{C}^{N+1})\setminus \mathcal{U}$ if and only if
\begin{align*}
\dim(E\cap L)\geq n-l+1\ .
 \end{align*}
 
 This motivates the following. Let $X^n\ra\cpn$. Fix $0\leq l \leq n$.   We define
 \begin{align*}
 Z:=
 \{E\in \mathbb{G}(N-(l+1),\cpn)\ | \  \exists \ p\in X\ \mbox{\emph{such that}}\ p\in E\ \mbox{\emph{and}}
 \ \dim(E\cap \mathbb{T}_p(X))\geq n-l\} \ .
 \end{align*}
In this situation we denote $Z$ by $Z_{l+1}(X)$. The reader should observe that 
\begin{align*}
& Z_{n+1}(X)=\{R_X=0\}\ \mbox{and}\ Z_1(X)=\{\Delta_X=0\}\ .
\end{align*}
 
$Z_n(X)$ plays a fundamental role in our study of the K-energy (see \cite{paul2012}) . In fact, we have that
\begin{align*}
Z_n(X)=\{\Delta_{\xhyp}=0\}\ ,
\end{align*}
the dual variety of the Segre image of  $\xhyp$.

In the situations of interest to us we may assume that $\mathcal{F}\cong W$ for an appropriate finite dimensional vector space $W$ . For example in \textbf{1} we may replace $\gr$ with $W=M_{(n+1)\times(N+1)}(\mathbb{C})$. 

In our applications $Z$ has \emph{codimension one} .  Then $Z\subset W$ is an irreducible algebraic hypersurface defined by a single polynomial $R_Z$
\[ Z=\{ f\in  W\ |\ R_Z(f)=0\}\ .\]
$Z$ is naturally dominated by the variety of zeros of a larger system $\{ s_j(p,f)=0\}$ in more variables $p\in X$. 
We define the \emph{incidence variety} $I_X$ by 
\[ I_X:=\{(p,\ f)\in X\times W\ | \ s_j(p\ ,\ f )=0\}\subset X\times W \]
In geometric terms $(p,f)\in I_X$ if and only if $f$ fails to meet $X$ generically at $p$ (and possibly at some other point $q$ ) .
Therefore, $Z$ is the \emph{resultant system} obtained by eliminating the variable $p$
\[ R_Z(f)=0 \ \mbox{iff}\ f\in Z \ \mbox{iff}\ \{s_j(\cdot \ ,\ f )=0\} \ \mbox{has a solution $p$ in $X$ for the fixed $f\in W$}\ . \]

\section{\ \ Linear Algebra of Complexes and the Torsion of a exact Complex}
To begin let $\left(E^{\bull},\dl^{\bull}\right)$ be a bounded complex of finite dimensional $\mathbb{C}$ vector spaces .
\[ \begin{CD} 
0@>>>E^{0}@>\dl_{0}>> \dots @>>> E^{i}@>\dl_{i}>> E^{i+1}@>\dl_{i+1}>> \dots @>\dl_{n}>> E^{n+1}@>>>0 \ . \end{CD}
\]
Recall that the \emph{determinant of the complex} $\left(E^{\bull},\dl^{\bull}\right)$ is defined to be the one dimensional vector space
\begin{align*}
  \mathbf{Det}(E^{\bull})^{(-1)^n}:=  \bigotimes_{i=0}^{n+1}(\bigwedge^{r_i}E^{i})^{(-1)^{i+1}}\quad r_i:=\mbox{dim}(E^i)\ .
\end{align*}
\begin{remark}  $\mathbf{Det}(E^{\bull})$ does not depend the boundary operators.
\end{remark}
As usual, for any vector space $V$ we set $V^{-1}:= \mbox{Hom}_{\mathbb{C}}(V,\mathbb{C})$, the dual space to
$V$.
Let $H^{i}(E^{\bull},\dl^{\bull})$ denote the $i^{th}$ cohomology group of this complex.
When $V = \mathbf{0}$ , the zero vector space , we set $\mathbf{det}(V):= \mathbb{C}$.
The determinant of the cohomology is defined in exactly the same way

\begin{align*}
\mathbf{Det}(H^{\bull}(E^{\bull},\dl^{\bull}))^{(-1)^n}:= \bigotimes_{i=0}^{n+1}(\bigwedge^{b_i}H^{i}(E^{\bull},\dl^{\bull}))^{(-1)^{i+1}}\quad b_i:= \mbox{dim}(H^{i}(E^{\bull},\dl^{\bull}))\ .
\end{align*}

\noindent We have the following well known facts (\cite{detdiv}, \cite{bgs1}). \\
 {\bf{D1}}
\emph{Assume that the complex} $\left(E^{\bull},\dl^{\bull}\right)$ \emph{is acyclic, then} $\mathbf{Det}(E^{\bull})$ \emph{is canonically trivial}
\begin{align*}\label{trivial}
\tau(\dl^{\bull}):\mathbf{Det}(E^{\bull}) \cong \underline{\mathbb{C}}\ .
\end{align*}
As a corollary of this we have,\\
{\bf{D2}}
 \emph{There is a canonical isomorphism \footnote{ A ``canonical isomorphism''  is one that only depends on the boundary operators, not on any choice of basis.} between the determinant of the complex and the}
\emph{determinant of its cohomology}:
\begin{align*}
\tau(\dl^{\bull}):\mathbf{Det}(E^{\bull}) \cong \mathbf{Det}(H^{\bull}(E^{\bull},\dl^{\bull}))\ .
\end{align*}
It is {\bf{D1}} which is relevant for our purpose. It says is that 
\emph{there is a canonically given \textbf{nonzero} element of} $\mathbf{Det}(E^{\bull})$, called the \emph{torsion} of the complex, provided this complex is exact.
The torsion is the essential ingredient in the construction of $X$-resultants (Cayley-Chow forms) and $X$-discriminants (dual varieties) . 
We recall the construction for more information see \cite{bgs1}.

 Define $\kappa_{i}:= \mbox{dim}(\dl_i E^{i})$, now choose $S_{i}\in \wedge^{\kappa_{i}}(E^{i})$ with $\dl_i{S_{i}}\neq 0$, then $\dl_i{S_{i}}\wedge S_{i+1}$ spans $\bigwedge^{r_{i+1}}E^{i+1}$ (since the complex is exact), that is

\begin{align*}
 \bigwedge^{r_{i+1}}E^{i+1}= \mathbb{C}\dl_i{S_{i}}\wedge S_{i+1} .
\end{align*}
With this said we define\footnote{ The purpose of the exponent $(-1)^n$ will be revealed in the next section.}
\begin{align*}
{{\mathbf{Tor}\left(E^{\bull},\dl^{\bull}\right)^{(-1)^{n}}:= (S_{0})^{-1}\otimes(\dl_0 S_{0}\wedge S_{1})\otimes (\dl_1S_{1}\wedge S_{2})^{-1}\otimes \dots \otimes (\dl_{n} S_{n})^{(-1)^{n}}}} \ .
\end{align*}
Then we have the following reformulation of {\bf{D1}}.\\ 
\noindent{\bf{D3}}
\begin{align}
 \mathbf{Tor}\left(E^{\bull},\dl^{\bull}\right) \mbox{\emph{is independent of the choices}}\ S_{i}.
 \end{align}

By fixing a basis $\{e^i_{1},e^i_{2}, \dots e^i_{r_{i}}\}$ in each of the terms $E_{i}$ we may associate to this \emph{based} exact complex a \emph{scalar}.
\begin{align*}
\mathbf{Tor}\left(E^{\bull},\dl^{\bull};\{e^i_{1},e^i_{2}, \dots, e^i_{r_{i}}\} \right)\in \mathbb{C}^*.
\end{align*}
 Which is defined through the identity:
\begin{align*}
\mathbf{Tor}\left(E^{\bull},\dl^{\bull}\right)= \mathbf{Tor}\left(E^{\bull},\dl^{\bull}; \{e^i_{1},e^i_{2}, \dots e^i_{r_{i}}\}\right)\mathbf{det}(\dots e^i_{1},e^i_{2}, \dots, e^i_{r_{i}}\dots).
\end{align*}
Where we have set
\begin{align*}
\mathbf{det}(\dots e^i_{1},e^i_{2}, \dots, e^i_{r_{i}}\dots)^{(-1)^n}:= (e^0_{1}\wedge \dots \wedge 
e^0_{r_{0}})^{-1}\otimes \dots \otimes  (e^{n+1}_{1}\wedge \dots \wedge e^{n+1}_{r_{n+1}})^{(-1)^{n}}.
\end{align*}
When we have fixed a basis of our exact complex (that is, a basis of each term in the complex) we will call $\mathbf{Tor}\left(E^{\bull},\dl^{\bull}; \{e^i_{1},e^i_{2}, \dots ,e^i_{r_{i}}\}\right)$ the Torsion of the \emph{based exact} complex. It is, as we have said, a \emph{scalar quantity}.  
\begin{remark}
In the following sections we often base the complex without mentioning it explicitly and in such cases we write (incorrectly) $\mathbf{Tor}\left(E^{\bull},\dl^{\bull}\right)$ instead of \newline
$\mathbf{Tor}\left(E^{\bull},\dl^{\bull}; \{e^i_{1},e^i_{2}, \dots ,e^i_{r_{i}}\}\right)$.
\end{remark}
 
 We have the following well known \emph{scaling behavior} of the Torsion, which we state in the next proposition. Since it is so important for us, we provide the proof, which is yet another application of the rank plus nullity theorem.
\begin{proposition} ( The degree of the Torsion as a polynomial in the boundary maps)
\begin{align}\label{scaling}
\begin{split}
& \mbox{\emph{deg}}\mathbf{Tor}\left(E^{\bull}, \dl^{\bull}\right)= (-1)^{n+1}\sum_{i=0}^{n+1}(-1)^{i}i\mbox{\emph{dim}}(E^i) \ .
\end{split}
\end{align}
\end{proposition}
\begin{proof}
Let $\mu\in \mathbb{C}^{*}$ be a parameter. Then 
\begin{align*}
\mathbf{Tor}\left(E^{\bull},\mu \dl^{\bull}\right)^{(-1)^n}&=(S_{0})^{-1}\otimes(\mu\dl_0 S_{0}\wedge S_{1})\otimes (\mu\dl_1S_{1}\wedge S_{2})^{-1}\otimes \dots \otimes (\mu\dl_{n} S_{n})^{(-1)^{n}}\\
\ \\
&=\mu^{\kappa_0-\kappa_1+\kappa_2-\dots+(-1)^n\kappa_n}\mathbf{Tor}\left(E^{\bull}, \dl^{\bull}\right)^{(-1)^n} \ .
\end{align*}
It is clear that
\begin{align*}
\kappa_0-\kappa_1+\kappa_2-\dots+(-1)^n\kappa_n=\sum_{i=0}^{n+1}(-1)^{i+1}i(\kappa_i+\kappa_{i-1})\quad (\kappa_{n+1}=\kappa_{-1}:=0)\ .
\end{align*}
Exactness of the complex implies that we have the short exact sequence 
 \[ \begin{CD} 
0@>>>\dl_{i-1}E^{i-1}@>\iota>>E^{i}@>\dl_{i} >> \dl_iE^{i}@>>>0 \ . \end{CD}
\]
Therefore $\kappa_i+\kappa_{i-1}=\dim(E_i)$ .
\end{proof}

\section{\ \ Fourier-Mukai Transforms and The Geometric Technique}
  
In this section we prove Theorem \ref{hilb-disc}. The method of proof follows the \emph{Geometric Technique}.  The author's understanding is that the technique is due to many mathematicians (see \cite{cay}, \cite{detdiv} , \cite{kempf76} , \cite{gkz}, \cite{ksz}). The author has also learned a great deal about the method from the monograph of J. Weyman \cite{weyman}, especially the ``basic set-up'' of chapter 5.  It seems that the application of the technique to complex differential geometry is new.

Let $X$ a complex variety. Let $(\mathcal{E}^{\bull}\ ;\ \delta^{\bull})$ be an exact (bounded) complex of locally free sheaves over $X$. The discussion in the previous section implies the following
\begin{proposition}\label{canonicallytrivial} The determinant line bundle of the complex admits a canonical nowhere vanishing section $\Delta$ 
\begin{align*}
\mbox{\textbf{\emph{Det}}}(\mathcal{E}^{\bull}\ ;\ \delta^{\bull})\overset{\Delta}{\cong}\mathbb{C} \ .
\end{align*}
 \end{proposition}
   We return to the situation described in the introduction. The object of study is a smooth, linearly normal, family $\mathbb{X}\ra B$ of relative dimension $n$  together with a rank $n+1$ locally free sheaf $\mathcal{Q}$ satisfying the requirements of a basic set up:
 \begin{align*}
\xymatrix{&\mathcal{E}\ar[d]\\  
	I \ar@{^{(}->}[r]^-{\iota} \ar[d] &  \mathbb{X}\times\mathbb{P}(W) \ar@/^1pc/[u]^{ {\Psi}} \ar[d]^{\pi}&\\  
	{Z}\ar@{^{(}->}[r]\ar[d]&B\times \mathbb{P}(W)  \ar[d]\\
	B\ar[r]&B }
\end{align*} 
 
 The basic set up exchanges a high codimension family $\mathbb{X}\ra B$ for a \emph{family of divisors} $Z_b\subset \mathbb{P}(W)$ parametrized by the same base
\begin{align*}
Z_b:=Z\cap \left(\{b\}\times \mathbb{P}(W)\right)\ ,\ b\in B\ .
\end{align*}
 Recall that each $Z_b$ is irreducible and moreover that the degree of $Z_b$ in $\mathbb{P}(W)$ is given by
 \begin{align*}
 \deg(Z_b)=\int_{\mathbb{X}_b}c_n(\mathcal{Q}) \ ,
 \end{align*}
 and therefore constant in $b\in B$.

 To facilitate the study of $Z$ we pass to the birationally equivalent $I $ which is much easier to deal with. More precisely we study the direct image of the structure sheaf of $I$ viewed as a coherent sheaf on $\mathbb{X}\times\mathbb{P}(W)$.

Recall that the Koszul complex $\left({K^{\bull}\mathcal{E}} \ ;\ \dl^{\bull}:= \lrcorner\ {\mathbf{s}}\right)$ associated to $(\mathcal{E},\mathbf{s}$) 
 \begin{align*}
\overset{\lrcorner\ {\mathbf{s}}}{\ra} \Lambda^{j+1}(\mathcal{E}^{\vee})\overset{\lrcorner\ {\mathbf{s}}}{\ra}\Lambda^{j}(\mathcal{E}^{\vee})\overset{\lrcorner\ {\mathbf{s}}}{\ra} \dots \overset{\lrcorner\ {\mathbf{s}}}{\ra} \mathcal{O}_{\mathbb{X}\times\mathbb{P}(W)}   
 \end{align*}
 resolves $\iota_{*}(\mathcal{O}_{I})$.
 The main player in this paper is the following  Fourier-Mukai transform between the bounded derived categories\footnote{In this paper $D^b(X)$ is   just notation for the set of bounded complexes of coherent sheaves on $X$. }
  \begin{align*}
 \Phi^{K^{\bull}\mathcal{E}}_{\mathbb{X}\mapsto B\times \mathbb{P}(W)}
 \left(\cdot\right):D^{b}(\mathbb{X})\ra D^{b}(B\times \mathbb{P}(W))\ 
 \end{align*}
 associated to the projections
 \begin{align*}
\xymatrix{&\quad \mathbb{X}\times \mathbb{P}(W)\ar[ld]_p\ar[rd]^\pi& \\
 \mathbb{X}&&\quad B\times \mathbb{P}(W)}
 \end{align*}

 Recall that this transform is defined by the formula
 \begin{align*}
 \Phi^{K^{\bull}\mathcal{E}}_{\mathbb{X}\mapsto B\times \mathbb{P}(W)}\left( F^{\bull}\right):=\mbox{\textbf{R}}^{\bull}\pi_{*}\left(\mbox{\textbf{L}}^{\bull}p^{*}(F^{\bull})\otimes K^{\bull}\mathcal{E}\right)\ 
 \end{align*}
 
 where $\textbf{R}^{\bull}\pi_{*}$ and ${\textbf{L}}^{\bull}p^{*}$ denote the usual derived functors.
 
  We remind the reader that by Grauert's Theorem of coherence of higher direct images that the complex
   \begin{align*}
   \Phi^{K^{\bull}\mathcal{E}}_{\mathbb{X}\mapsto B\times \mathbb{P}(W)}\left( F^{\bull}\right)
   \end{align*}
   
   is represented by a bounded complex of (quasi) coherent $\mathcal{O}_{B\times \mathbb{P}(W)}$-modules with {coherent} cohomology sheaves.\\
  
 We are interested in the value of this transform at a particular point of $D^{b}(\mathbb{X})$
  \begin{align*}
   \Phi^{K^{\bull}\mathcal{E}}_{\mathbb{X}\mapsto B\times \mathbb{P}(W)}\left( \mathcal{O}_{\mathbb{X}}(m) \right)\in D^{b}\left( B\times \mathbb{P}(W)\right) \quad m>> 0\ .
   \end{align*}
  
 \ \\
  Since the (twisted) Koszul complex resolves $\iota_{*}\mathcal{O}_{I}\otimes p^{*}\mathcal{O}_{\mathbb{X}}(m) $ we see that \emph{the cohomology sheaves of the complex} $$ \Phi^{K^{\bull}\mathcal{E}}_{\mathbb{X}\mapsto B\times \mathbb{P}(W)}\left( \mathcal{O}_{\mathbb{X}}(m) \right)$$ \emph{are all supported on} $Z$.
  
 Take $m$ large enough to force the term wise vanishing of all higher direct image sheaves 
   \begin{align*}
   \mbox{\textbf{R}}^{i}\pi_{*}\left(\Lambda^{j}(\mathcal{E}^{\vee})\otimes p^{*}\mathcal{O}_{\mathbb{X}}(m) \right)=0\qquad i>0\ \ \mbox{all $j$} \ .
   \end{align*}
 
 This has the crucial implication that the natural map
 \begin{align*}
\pi_{*}\left( \Lambda ^{\bull}\mathcal{E}^{\vee}\otimes p^{*}\mathcal{O}_{\mathbb{X}}(m) \ , \ \dl^{\bull}\right)\ra \Phi^{K^{\bull}\mathcal{E}}_{\mathbb{X}\mapsto B\times \mathbb{P}(W)}\left( \mathcal{O}_{\mathbb{X}}(m) \right)
 \end{align*}
 is the identity in $D^{b}\left( B\times \mathbb{P}(W)\right)$ , in other words, the two complexes are quasi-isomorphic.  This is a consequence of the following lemma.    
\begin{lemma}\label{hypercohomology} 
 Let $\mathfrak{A}$ and $\mathfrak{B}$ be abelian categories. Assume that $\mathfrak{A}$ has enough injectives. Let $F:\mathfrak{A}\ra \mathfrak{B}$ be a left exact functor. Let $(N^{\bull}\ , \ \delta_{\bull})$ denote a bounded below complex of elements of $\mathfrak{A}$ where each term is $F$-acyclic
 \begin{align*}
 \mathbf{R}^iF(N^{j})=0  \qquad i>0\ \ \mbox{all $j$} \ .
 \end{align*}
 Then for any injective resolution 
 $$(N^{\bull}\ , \ \delta_{\bull})\overset{\iota_{\bull}}{\ra} (I^{\bull} \ , \ \dl_{\bull})$$
  the induced map of complexes
  \begin{align*}
   (F(N^{\bull})\ , \ F(\delta_{\bull}))\overset{F(\iota_{\bull})\quad }{\ra}  (F(I^{\bull}) \ , \ F(\dl_{\bull})) 
   \end{align*}
 is a quasi-isomorphism. 
 \end{lemma}
 \begin{proof}
 See Lemma 8.5 pg.188 of \cite{voisinI} .
 \end{proof}
 \begin{remark}
 \emph{By definition, $(F(I^{\bull}) \ , \ F(\dl_{\bull}))$ is the $F$-derived image of 
 $(N^{\bull}\ , \ \delta_{\bull})$} .  
 \end{remark}
 \ \\

Lemma \ref{hypercohomology} allows us to replace the apriori quite complicated object  
 \begin{align*}
 \Phi^{K^{\bull}\mathcal{E}}_{\mathbb{X}\mapsto B\times \mathbb{P}(W)}\left( \mathcal{O}_{\mathbb{X}}(m) \right)
 \end{align*}
  with the much simpler (termwise) direct image complex 
 $$\pi_{*}\left( \Lambda ^{\bull}\mathcal{E}^{\vee}\otimes p^{*}\mathcal{O}_{\mathbb{X}}(m) \ , \ \dl^{\bull}\right)\ . $$

The purpose of the foregoing discussion was to put us in the following situation
 \begin{proposition}
  The complex of locally free sheaves over $B\times \mathbb{P}(W)$
 \begin{align*}
\pi_{*}\left( \Lambda ^{\bull}\mathcal{E}^{\vee}\otimes p^{*}\mathcal{O}_{\mathbb{X}}(m) \ , \ \dl^{\bull}\right)
  \end{align*}
 is exact away from $Z$.
 \end{proposition}
 
 Therefore away from $Z$ the determinant of the direct image complex has a nowhere vanishing section   
 \begin{align*}
\xymatrix{&{\mathbf{Det}}\ \pi_{*}\left( \Lambda ^{\bull}\mathcal{E}^{\vee}\otimes p^{*}\mathcal{O}_{\mathbb{X}}(m) \ , \ \dl^{\bull}\right)\ar[d]\\
 & B\times \mathbb{P}(W)\setminus Z \ar@/^1pc/[u]^{\Delta}} 
  \end{align*}
  
  \begin{align*}
 \Delta:=\mbox{\textbf{Tor}}\ \pi_{*}\left( \Lambda ^{\bull}\mathcal{E}^{\vee}\otimes p^{*}\mathcal{O}_{\mathbb{X}}(m) \ ,\ \dl^{\bull}\right) \ .
\end{align*}  

\begin{proposition}
There is an invertible sheaf $\mathcal{A}$ over $B$ such that 
\begin{align*}
 \mathbf{Det}\ \pi_{*}\left( \Lambda ^{\bull}\mathcal{E}^{\vee}\otimes p^{*}\mathcal{O}_{\mathbb{X}}(m) \ , \ \dl^{\bull}\right)\cong p_{1}^*\mathcal{A}\otimes p_{2}^{*}\mathcal{O}_{W}(\widehat{d})  \ .
 \end{align*}
 \end{proposition}
\begin{proof} 
 The proof is quite easy.
 Observe that by the projection formula, the stalk of the direct image of our Koszul complex at $b\in B$ is given by
 \begin{align*}
\pi_{*}\left( \Lambda ^{j}\mathcal{E}^{\vee}\otimes p^{*}\mathcal{O}_{\mathbb{X}}(m) \right)|_{b}\cong H^{0}\left(\mathbb{X}_b\ , \ \Lambda^{j} Q(m)\right)\otimes\mathcal{O}_{W}(-j) \ .
 \end{align*}

  We define
\begin{align*}
r_{j}(m):=\dim H^{0}\left(\mathbb{X}_b\ , \ \Lambda^{j} Q(m)\right) \ .
\end{align*}

 Then we define $\mathcal{A}$ as follows
\begin{align*}
\mathcal{A}|_{b}:= \bigotimes _{0\leq j\leq n+1}\Lambda^{r_{j}(m)} H^{0}\left(\mathbb{X}_b\ , \ \Lambda^{j} Q(m)\right)^{(-1)^{j+1}}  \ .
\end{align*}
Therefore the determinant is given by
\begin{align*}
{\mathbf{Det}}\ \pi_{*}\left( \Lambda ^{\bull}\mathcal{E}^{\vee}\otimes p^{*}\mathcal{O}_{\mathbb{X}}(m) \ , \ \dl^{\bull}\right)\cong p_{1}^{*}\mathcal{A}\otimes \mathcal{O}_{W}(\chi) \ ,
\end{align*}
where we have defined $\chi\in \mathbb{Z}$ by
 \begin{align*}
 \chi:=\sum_{0\leq j\leq n+1}(-1)^{j+1}jr_{j}(m)\in \mathbb{Z} \ .
 \end{align*}
  
So the argument comes down to showing that $\chi=\widehat{d}$ 
which a straightforward but tedious calculation with the Hirzebruch-Riemann-Roch Theorem will verify.
\end{proof}
Fix $b\in B$. Let $f_b$ denote a defining polynomial of $Z_b$. Since $\Delta|_{\{b\}\times \mathbb{P}(W)}$ is without zeros or poles away from $Z_b$ there is an integer $$\mbox{ord}_{Z_b}(\Delta|_{\{b\}\times \mathbb{P}(W)})$$
satisfying
\begin{align*}
\Delta|_{\{b\}\times \mathbb{P}(W)}=f_b^{\mbox{ord}_{Z_b}(\Delta|_{\{b\}\times \mathbb{P}(W)})} \ .
\end{align*}
 Our computation of the degree of the torsion shows that  
\begin{align*} 
 \deg(\Delta|_{\{b\}\times \mathbb{P}(W)})=\chi=\widehat{d}  \ . 
\end{align*}
 Therefore $\mbox{ord}_{Z_b}(\Delta|_{\{b\}\times \mathbb{P}(W)})=1$ , and we have shown
\begin{align*}
\mathbb{C}^*\Delta|_{\{b\}\times \mathbb{P}(W)}=\mathbb{C}^*f_{b}  \ .
\end{align*}
Therefore we have the following
\begin{proposition} 
$\Delta$ vanishes on $Z$ and in particular extends to a global section of the determinant line.  
 \end{proposition}

  Therefore, to any basic set up for the family $\mathbb{X}\ra B$ we may associate the following\\
 \ \\
\noindent$\bull$ An invertible sheaf $\mathcal{A}\in \mbox{Pic}(B)$ . \\
\ \\
$\bull$ An algebraic section 
 \begin{align*}
 \Delta\ \in H^{0}\left(B\times\mathbb{P}(W) \ , \ p_1^*\mathcal{A}\otimes p_2^*\mathcal{O}_{W}( \widehat{d})\right) \ .
 \end{align*}
$\bull$  A relative Cartier divisor $Z$ over $B$
 \begin{align*}
 {Z}:= \mbox{Div}(\Delta)\subset B\times \mathbb{P}(W) \ . \\
 \end{align*}
 $\bull$ Moreover, the direct image $T$ of $\Delta$ \\
\begin{align*}
 \xymatrix{   \mathcal{A}\otimes  H^{0}(\mathbb{P}(W) , \mathcal{O}_{W}( \widehat{d}) )  \ar[d] \\
B\ar@/^1pc/[u]^{T:= {p_1}_*(\Delta)}}
\end{align*}
 {never vanishes on $B$.}
 
\ \\

This package induces a morphism (also denoted by $\Delta$) from $B$ to the complete linear system on $\mathbb{P}(W)$ 
\begin{align*}
\Delta:B\ra |\mathcal{O}_{W}( \widehat{d})| 
\end{align*}
 as follows . 
Since $T\neq 0$ , $T$ gives an injection
\begin{align*}
0\ra\mathcal{O}_B\xrightarrow{\times T} \mathcal{A}\otimes H^{0}(\mathbb{P}(W),\mathcal{O}( \widehat{d})) \ .
\end{align*}
Dualizing and tensoring with $\mathcal{A}$ gives a surjection  
\begin{align*}
H^{0}(\mathbb{P}(W),\mathcal{O}( \widehat{d}))^{\vee}\times B \ra \mathcal{A}\ra 0 \ ,
\end{align*}
hence $\mathcal{A}$ is globally generated.

Therefore we obtain a morphism as required. This completes the proof of Theorem \ref{hilb-disc}. Moreover we see that
\begin{align*}
\Delta^*\mathcal{O}_{|\mathcal{O}( \widehat{d})|}(1)\cong \mathcal{A} \ ,
\end{align*}
  and the natural map
  \begin{align*}
   H^{0}(\mathbb{P}(W),\mathcal{O}( \widehat{d}))^{\vee}\ra H^0(B, \mathcal{A}) \\
   \end{align*}
is an injection. The map exhibits a large (generating) finite dimensional subspace of the space of sections of $\mathcal{A}$ over $B$.  
Conversely, given such a map $\Delta$, we define 
 \begin{align*}
 &\mathcal{A}:=\Delta^*\mathcal{O}_{|\mathcal{O}( \widehat{d} )|}(1)  \\
 \ \\
 &\Delta(b,[w]):=\Delta(b)([w])\in\mathcal{A}_b\otimes \mathcal{O}_{W}( \widehat{d}) \ . 
 \end{align*}
 Next we give several examples of basic set ups for a given family $\mathbb{X}\ra B$.
  \subsection{The basic set up for Resultants} 
 Let $\mathbb{X}\rightarrow  B$ be a flat family of polarized subvarieties of $\mathbb{P}^N$.  We can arrange a basic set up 
\begin{align*}
\xymatrix{&\mathcal{E}\ar[d]\ar[r]& p_1^*\mathcal{O}_{\cpn}(1)\otimes p_2^*\mathcal{Q}\ar[d] \\
 I \ar@{^{(}->}[r]^-{\iota} \ar[d] &  \mathbb{X}\times\mathbb{G} \ar@/^1pc/[u]^{\Psi}\ar[r]^-{\pi_2\times 1}\ar[d]^{\pi}&\cpn\times \mathbb{G}  \ar@/^1pc/[u]^{\Pi}\\
 {Z}\ar@{^{(}->}[r]&B\times \mathbb{G}  \ar[d]^{p_1}&\\
 &B& }
\end{align*}
for this family if we define
 \begin{align*}
&\bull \ \Pi|_{([v],L)}:\mathbb{C}v\ra \mathbb{C}^{N+1}/L \ , \ \quad \Pi(zv)=\pi_{L}(zv) \\
& \quad \pi_L:\mathbb{C}^{N+1}\ra \mathbb{C}^{N+1}/L \ \mbox{denotes the projection} . \\
&\quad \Psi:=(\pi_2\times 1)^*\Pi\ . \\
 \ \\
&\bull\   \mathbb{G}:=\mathbb{G}(N-n-1,N) \ .\\
\ \\
&\bull \ \mathcal{E}:=(\pi_2\times 1)^*\left(p_1^*\mathcal{O}_{\cpn}(1)\otimes p_2^*\mathcal{Q}\right) \\
&\ \ \quad I = (\Psi=0)  \ . \\
 \ \\
& \bull\ \pi:\mathbb{X}\times \mathbb{G}\ra B\times \mathbb{G} \quad \mbox{is defined by\footnote{Abuse of notation.}} \ \pi(x,L):=(\pi(x),L)\\
& \quad  {Z}:=\pi( I)\ . \\
\ \\
&\bull\ Z_b \ \mbox{is the \emph{\textbf{Cayley form}} of $X_b$ and $\Delta_b$ is the $X_b$-\emph{\textbf{resultant}}}\ .
\end{align*}
 
\subsection{The basic set up for Discriminants}
In this section we consider a flat family $\mathbb{X}\xrightarrow{\pi} B$ of polarized manifolds.  
The basic set up that we consider in this case has the form
\begin{align*}
\xymatrix{&\mathcal{E}\ar[d]\ar[r]& p_1^*\mathscr{U}^{\vee}\otimes p_2^*\mathcal{O}_{\check{\mathbb{P}}^N}(1)\ar[d] \\
 \Gamma_{\mathbb{X}}\ar@{^{(}->}[r]^-{\iota} \ar[d] &  \mathbb{X}\times\check{\mathbb{P}}^N \ar@/^1pc/[u]^{\mathbf{s}}\ar[r]^-{\rho\times 1}\ar[d]^{p}&\mathbb{G}(n,N)\times \check{\mathbb{P}}^N  \ar@/^1pc/[u]^{\Lambda}\\
 {Z}\ar@{^{(}->}[r]&B\times \check{\mathbb{P}}^N  \ar[d]^{p_1}&\\
 &B& }
\end{align*}
In the diagram above we have defined
\begin{align*}
&\bull \ \Lambda|_{(L,[f])}:L\ra \mathbb{C}^{N+1}/\ker(f) \ , \ \quad \Lambda(u)=\pi_{\ker(f)}(u) \\
& \quad \pi_{\ker(f)}:\mathbb{C}^{N+1}\ra \mathbb{C}^{N+1}/{\ker(f)} \ \mbox{denotes the projection} . \\
&\quad \mbox{Observe that}\  \Lambda|_{(L,[f])}=0 \ \mbox{if and only if}\ L\subset \ker(f)\ . \\
 \ \\
&\bull \ \rho=\rho_{\mathbb{X}}:\mathbb{X} \ra \mathbb{G}(n,N)   \ \mbox{is the fiber wise Gauss map.} \\ 
&\quad \mathscr{U} \ \mbox{is the tautological bundle}\ .\\
\ \\
&\bull \ \mathcal{E}:=(\rho\times 1)^*\left( p_1^*\mathscr{U}^{\vee}\otimes p_2^*\mathcal{O}_{\check{\mathbb{P}}^N}(1)     \right) \\
&\ \ \quad \Gamma_{\mathbb{X}}:= (\mathbf{s}=0) \ ;  \   {Z}:=q( \Gamma_{\mathbb{X}})\ .\\
 \ \\
& \bull Z_b\ \mbox{is the \emph{\textbf{dual variety}} of $X_b$ and $\Delta_b$ is the $X_b$-\emph{\textbf{discriminant}}} \ .
 \end{align*}
 \ \\
 \begin{remark}
 \emph{An interesting generalisation of these two examples is constructed as follows. Given a family $\mathbb{X}\ra B$ let $\mathscr{E}_{k}$ denote the rank $n-k+1$ trivial bundle over $\mathbb{X}$. Then we obtain a new family}
 \begin{align*}
 \mathbb{P}(\mathscr{E}_{k})\ra B \ .
 \end{align*}
\emph{The fiberwise Segre embedding exhibits this family as a family of subvarieties of the projective space of matrices of size $(N+1)\times (n-k+1)$.
If the original family is smooth one may apply the set up for discriminants to this new family. When $k=1$ the corresponding polynomial
$\Delta_b$ is the $X_b$-\emph{\textbf{hyperdiscriminant}}.}
\end{remark}
\section{\ \ Comparing the currents $\delta_Z$ and $\delta_{I}$ over $S$}

Now we prove Corollary \ref{metrics}. To begin, let $S\ra B$ be a morphism from a \emph{smooth} \footnote{Smoothness is required to apply the Poincar\'e Lelong formula.} variety $S$.
As stated in the introduction we assume 
\begin{align*}
\mathcal{A}\cong \mathcal{O}_{S}
\end{align*}
and there exists a smooth map  
 \begin{align*}
h:S\ra \mathcal{H}^{+}(W, <\cdot\ ,\ \cdot>)
\end{align*}
satisfying conditions $A_1$ and $A_2$ .
 There is an induced Hermitean metric on $h$ on the determinant line bundle  $$p_1^*\mathcal{A}\otimes p_2^*\mathcal{O}_{W}( \widehat{d})$$ and it is not hard to see that the square of the length of our section $\Delta$ is given by
\begin{align*}
\frac{|\Delta( s)([w])|^2_{h^{\widehat{d}}}}{||\Delta (s)||_2^2}
\end{align*}
The denominator being the usual $L^2$ norm
\begin{align*}
||\Delta (s)||_2^2:=\int_{\mathbb{P}(W)}|\Delta(s)(\cdot)|^2_{h^{\widehat{d}}}\ \om_{FS}^{l}\quad ;\ l+1=\dim(W)\ .
\end{align*}

The Poincar\'{e} Lelong formula gives
\begin{align*}
\frac{\sqrt{-1}}{2\pi}\dl\dlb\log\frac{|\Delta(s)( [w])|^2_{h^{\widehat{d}}}}{||\Delta (s) ||_2^2}=\delta_{Z}-c_{1}\left( \mathcal{A}\otimes  \mathcal{O}_{W}( \widehat{d})\ ; \ h\right) \ .
\end{align*}
   Recall that the triviality of $\mathcal{A}$ over $S$ is equivalent to having a lift of \footnote{We will also denote the lifted map by $\Delta$ as well.} the map $\Delta$  to the affine cone
\begin{align*}
\Delta:S\ra  H^{0}(\mathbb{P}(W),\mathcal{O}( \widehat{d})) \setminus \{0\} \ .
\end{align*}
Next fix some base point $o\in S$. Then
\begin{align*}
c_{1}\left( \mathcal{A}\otimes  \mathcal{O}_{W}( \widehat{d})\ ; \ h\right)|_{S\times\mathbb{P}(W)}=\widehat{d}\om_{FS}+\frac{\sqrt{-1}}{{2\pi}}\dl\dlb\log\frac{||\Delta(s)||_2^2}{||\Delta(o)||_2^2}
\end{align*}
With this said we have the following proposition concerning the direct image of this current under $\pi$.
\begin{proposition}\label{poincare-lelong}
Let $\eta$ be a smooth compactly supported form on $S$
\begin{align*}
\eta\in C^{\infty}\Lambda_0 ^{\dim(S)-1 ,\dim(S)-1}(S)
\end{align*}
Then  
\begin{align*}
\int_{Z}\pi^{*}(\eta)\wedge\om^l_{FS}=\int_{S}\eta\wedge \frac{\sqrt{-1}}{{2\pi}}\dl\dlb\log\left(e^{\theta(s)} \frac{||\Delta(s)||_2^2}{||\Delta( o)||_2^2}\right)
\end{align*}
where $\theta$ is defined by
\begin{align*}
\theta(s):= \int_{\mathbb{P}(W)}\log\frac{|\Delta( s)( [w])|^2_{h^{\widehat{d}}}}{||\Delta (s)||_2^2}\om^l_{FS}\ .
\end{align*}
\end{proposition}
 
 The birationality of  $I $ and $Z$ imply that we have the identity 
\begin{align*}
\int_{Z}\pi^{*}(\eta)\wedge\om^l_{FS}=\int_{I }\pi^*(\eta)\wedge\om^l_{FS} 
\end{align*}
 for all compactly supported forms $\eta$ on $S$. Property A1 of the current $\widehat{\mathbf{e}}(H_0^{\mathcal{E}})$ implies that
 \begin{align*}
 \int_{Z}\pi^{*}(\eta)\wedge\om^l_{FS}=
\int_{\mathbb{X}\times\mathbb{P}(W)}\frac{\sqrt{-1}}{2\pi}\dl\dlb{\Psi}^{*}\widehat{\mathbf{e}}(H_0^{\mathcal{E}})\wedge\om^l\wedge \pi^{*}(\eta) \ .
 \end{align*}
 
Property A2 and the variation formula show that 
 \begin{align*}
 \int_{\mathbb{X}_{s}\times\mathbb{P}(W)}{\Psi}^{*}\widehat{\mathbf{e}}(H_{0}^{\mathcal{E}})\wedge\om^l=\int_{\mathbb{X}_{s}\times\mathbb{P}(W)}\textbf{e}(H^{\mathcal{E}}_{0};{H}^{\mathcal{E}}_{S})\wedge \om^l+\mbox{a pluriharmonic function of $s$} \ .
 \end{align*}
   
Therefore we see that for all compactly supported forms $\eta$ we have
 \begin{align*}
\int_{S}  \eta\wedge \frac{\sqrt{-1}}{2\pi}\dl\dlb\int_{\mathbb{X}_s\times\mathbb{P}(W)}  Ò{\mathbf{e}}(H^{\mathcal{E}}_{0};{H}^{\mathcal{E}}_{S})\wedge \om^l =\int_{S}\eta\wedge \frac{\sqrt{-1}}{{2\pi}}\dl\dlb\log\left(e^{\theta(s)} \frac{||\Delta(s)||_2^2}{||\Delta( o )||_2^2}\right)\ .
\end{align*}
Therefore we have proved Corollary \ref{metrics} .

Given a fixed projective variety $X\subset \cpn$ with a non-degenerate dual variety $\check{X}$ we construct, following \cite{tian97}, a ``tautological family'' $\mathbb{X}\ra G$ where $G=\slnc$ with fiber $\sigma X$. Then, for each $0\leq k\leq \dim(X)$ , we consider the new family $\mathbb{P}(\mathscr{E}_k)\ra G$. Let $\mathcal{Q}$ denote the vector bundle associated to the set up for discriminants for this family. Then for the obvious choice of Hermitean inner product on $W$, there is a natural map 
\begin{align*}
h :G\ra \mathcal{H}^+(W , <\cdot ,\cdot>) \ 
\end{align*}
satisfying A1 and A2 with basepoint $o=e$, the corresponding action(s) $E_{\om}$ \footnote{This is forthcoming work of the first author and his student Q. Westrich.} are related to the higher K-energy maps of Mabuchi.

\bibliographystyle{plain} 
\bibliography{ref.bib}
\end{document}